\documentclass[12pt]{amsart}
\usepackage{amsmath, amsthm, amscd, amsfonts, amssymb, graphicx, color}
\usepackage[bookmarksnumbered, colorlinks, plainpages]{hyperref}
\usepackage{plain}

\def\rd{\mathbb{R}^d}
\renewcommand{\epsilon}{\varepsilon}

\newcommand{\rarw}{\rightarrow}

\newcommand\pot{^}
\newcommand\ubd{\overline{\mbox{\rm dim}}_{\rm B}\,} 
\newcommand\lbd{\underline{\mbox{\rm dim}}_{\rm B}\,} 
\newcommand\da{\underline{\mbox{\rm dim}}_{\,\theta}} 
\newcommand\uda{\overline{\mbox{\rm dim}}_{\,\theta}} 
\newcommand{\dimh}{\dim_{\rm H}}

\newcommand{\be}{\begin{equation}} 
	\newcommand{\ee}{\end{equation}}

\newtheorem{theo}{Theorem}[section]
\newtheorem*{theo*}{Theorem}

\newtheorem*{cor*}{Corollary}

\newtheorem{prop}[theo]{Proposition}

\newtheorem{defi}[theo]{Definition}

\newtheorem{question}{Question}[section]

\title {Two-Scale Frostman Measures}
\author{Nicolás Angelini}
\address{Departamento de Matemática, Facultad de Ciencias Físico Matemáticas y Naturales, Universidad Nacional de San Luis,\\
and Instituto de Matemática Aplicada San Luis (IMASL, CONICET), San Luis, Argentina.}
\email{nicolas.angelini.2015@gmail.com}

\author{Úrsula Molter}
\address{Departamento de Matemática, Facultad de Ciencias Exactas y Naturales, Universidad de Buenos Aires,\\
and Instituto de Investigaciones Matemáticas Luis A. Santaló (IMAS, UBA-CONICET), Buenos Aires, Argentina.}
\email{umolter@gmail.com}

\subjclass[2020]{28A80, 28A75}
\keywords{Fractals, Frostman measures,Intermediate dimensions}

\begin{document}

\begin{abstract}
We establish a unified Frostman-type framework connecting the classical Hausdorff dimension with the family of intermediate dimensions $\dim_\theta$ recently introduced by Falconer, Fraser and Kempton.
We define a new geometric quantity $\mathcal{D}(E)$ and prove that, under mild assumptions, there exists a family of measures $\{\mu_\delta\}$ supported on $E$ satisfying two simultaneous decay conditions, corresponding to the Hausdorff and intermediate Frostman inequalities.
Such $(\delta, s, t)$-Frostman measures allow for a two-scale characterization of the dimension of $E$.
\end{abstract}
\maketitle

\section{Introduction}

The study of fractal dimensions provides a quantitative framework for describing the fine-scale geometry of sets across multiple scales. A unifying theme in this area is the characterization of dimension through the decay of measures supported on the set.

The classical example of this correspondence is Frostman’s lemma, which relates the Hausdorff dimension of a closed set \(E \subset \mathbb{R}^n\) to the polynomial decay of measures supported on it. Specifically, Frostman \cite{frostman1935potential} proved that \( \mathcal{H}^s(E) > 0 \) if and only if there exists a Borel probability measure \(\mu\) supported on \(E\) such that 
\begin{equation}\label{Hausdorff}
    \mu(B(x,r)) \le c\, r^s 
    \quad \text{for all } x \in \mathbb{R}^n \text{ and } r>0,
\end{equation}
for some constant \(c>0\). Here, \(\mathcal{H}^s\) denotes the \(s\)-dimensional Hausdorff measure.

More recently, Falconer, Fraser, and Kempton \cite{FFK} introduced a one-parameter family of \emph{intermediate dimensions}, denoted by the lower and upper quantities \(\underline{\dim}_\theta E\) and \(\overline{\dim}_\theta E\) for \(\theta \in [0,1]\). These dimensions interpolate continuously between the Hausdorff and box-counting dimensions, capturing a broader range of scaling behaviors by restricting the admissible diameters in the definition of the Hausdorff dimension to lie within a geometric range determined by \(\theta\).

The study of intermediate dimensions has developed rapidly, and several recent works have explored their geometric and measure-theoretic properties. For instance, projection theorems have been investigated in \cite{BFF,angelini2025critical}, attainable forms in \cite{BanajiRutar2022,angelini2024intermediate}, and images of more general functions in \cite{BFF2}. In parallel, various generalizations and alternative definitions of measures have been proposed \cite{Banaji,ANGELINI2026130039}.

In analogy with the classical Hausdorff dimension, intermediate dimensions can also be characterized in terms of the polynomial decay of measures supported on \(E\). The following proposition provides a Frostman-type characterization for these dimensions.

\begin{prop}\label{FrostmanI}
Let \(E \subset \mathbb{R}^n\) be compact and \(\theta \in (0,1]\).
\begin{enumerate}
    \item If \(\underline{\dim}_\theta E > 0\), then for every \(s \in (0, \underline{\dim}_\theta E)\) there exists a constant \(c>0\) such that, for every \(\delta \in (0,1)\), one can find a Borel probability measure \(\mu_\delta\) supported on \(E\) satisfying
    \[
        \mu_\delta(B(x,r)) \le c\, r^s 
        \quad \text{for all } x \in E \text{ and } \delta^{1/\theta} \le r \le \delta.
    \]
    \item If \(\overline{\dim}_\theta E > 0\), then for every \(s \in (0, \overline{\dim}_\theta E)\) there exists a constant \(c>0\) such that, for every \(\delta_0 > 0\), there exist \(\delta \in (0, \delta_0)\) and a Borel probability measure \(\mu_\delta\) supported on \(E\) satisfying
    \begin{equation}\label{FrostmanIneq}
        \mu_\delta(B(x,r)) \le c\, r^s 
        \quad \text{for all } x \in E \text{ and } \delta^{1/\theta} \le r \le \delta.
    \end{equation}
\end{enumerate}
\end{prop}

\medskip

In this paper, we introduce a new geometric quantity \(\mathcal{D}(E)\) and establish sufficient conditions ensuring the existence of a family of measures \(\{\mu_\delta\}\) that simultaneously satisfy both decay estimates \eqref{Hausdorff} and \eqref{FrostmanIneq}. This provides a unified framework linking the classical Frostman lemma with its intermediate counterpart.

\begin{theo*}
    Let $\theta > 0$ and $E \subset \mathbb{R}\pot{d}$ be a compact set such that $0 < \mathcal{D} (E)$. 
    Then for all $0 < t < \uda E$ and all $0 < s < \mathcal{D}(E)$ with $s \le t$, there exists $c>0$ such that for all $\delta_0 > 0$ there exist $\delta \in (0, \delta_0)$ and a Radon measure $\mu_{\delta}$ supported on $E$ satisfying
    \begin{equation}\label{aaclarar3}
        \mu_{\delta}(B(x,r)) \le 
        \begin{cases}
            c\, (\delta^{1/\theta})^{t-s} r^{s}, & \text{if } r \in (0, \delta^{1/\theta}),\\[4pt]
            c\, r^{t}, & \text{if } r \in [\delta^{1/\theta}, \delta].
        \end{cases}
    \end{equation}
    for all $x \in E$.
    
    Equivalently, if $0< t < \da E$ and $0<s<\mathcal{D}(E)$, there exist constants $c>0$ and $\delta_0>0$ such that for all $\delta \in (0, \delta_0]$ there exists a measure $\mu_\delta$ supported on $E$ satisfying \eqref{aaclarar3}.
\end{theo*}

A Borel probability measure satisfying \eqref{aaclarar3} will be called a \((\delta, s, t)\)-Frostman measure.

\section{Preliminaries}\label{preliminares}

Throughout the whole document $\mathcal{D}_n$ will denote the classical dyadic partition of $\mathbb{R}\pot{d}$ into $2\pot{dn}$ half-open disjoint cubes of diameter $\sqrt{d}\,2\pot{-n}$. 

$B(x,r)$, $r>0$ will denote the open ball in $\rd$ with center $x$ and radio $r$ and $\mathcal{B}$ the Borel subsets of $\rd$. 

Given a non-empty set $A\subset\rd$ we will write $|A|$ for the diameter of the set and the distance between $A$ and a point $x$ will be $d(x,A)=\displaystyle\inf_{y\in A}|x-y|$.

The Hausdorff dimension of $A$ will be $\dimh A$ and $\dim_B A$ will denote Box-counting dimension of $A$ . The reader can refer to \cite{FBook} for more information on these dimensions.

$\mathcal{H}\pot{s}$ will always refer to the Hausdorff $s-$measure and $\mathcal{L}\pot{d}$ the $d-$dimensional Lebesgue measure.

The support of a measure $\mu$ on $\rd$ is defined as $spt\,\mu=\rd\setminus\{x:\exists\, r>0\mbox{ such that } \mu(B(x,r))=0\}$.

Throughout the paper Radon measure will mean a locally finite and Borel regular measure.

This document concerns the $\theta$ intermediate dimensions, which were introduced in \cite{FFK} and are defined as follows.

\begin{defi}\label{adef}
Let $F\subseteq \mathbb{R}\pot{d}$ be bounded. For $0\leq \theta \leq 1$ we define the {\em  lower $\theta$-intermediate dimension} of $F$ by
\begin{align*}
 \da F =  \inf \big\{& s\geq 0  : \ \forall\  \epsilon >0, \ {\rm and\ all }\  \delta_0>0, {\rm there}\ \exists\ 0<\delta\leq \delta_0, \ { \rm and }\  \\
 & \{U_i\}_{i\in I}\  : F \subseteq \cup_{i\in I}U_i \ : \delta^{1/\theta} \leq  |U_i| \leq \delta \ {\rm and }\ \sum_{i\in I} |U_i|^s \leq \epsilon  \big\}.
\end{align*}

Similarly, we define the {\em  upper $\theta$-intermediate dimension} of $F$ by
\begin{align*}
\uda F =  \inf \big\{& s\geq 0  : \ \forall\  \epsilon >0, \ {\rm there }\ \exists\ \delta_0>0, \ : \ \forall\ 0<\delta\leq \delta_0, \mbox{ \rm there }\  \exists \\
 & \{U_i\}_{i\in I}\  : F \subseteq \cup_{i\in I}U_i \ : \delta^{1/\theta} \leq  |U_i| \leq \delta \ {\rm and }\ \sum_{i\in I} |U_i|^s \leq \epsilon  \big\}.
\end{align*}
\end{defi}

For all $\theta\in (0,1]$ and a compact set $A\in\rd$ we have $\dimh A\leq \uda A\leq \ubd A$ and similarly with the lower case replacing $\ubd$ with $\lbd$. This spectrum of dimensions has the property of being continuous for $\theta\in(0,1]$, leaving the natural question of the continuity in $\theta=0$ as a problem to study. The reader can find more information in \cite{FFK}.

Finally, the lower dimension of a non-empty set $E\subset \mathbb{R}^d$ is defined as

\begin{align*}
\dim_L E = \sup \Big\{ 
&\alpha : \text{ there are constants } c, \rho > 0 \text{ such that} \\[4pt]
&\inf_{x \in E} N_r(B(x, R) \cap E) \geq c \left( \frac{R}{r} \right)^{\alpha}
\text{ for all } 0 < r < R < \rho 
\Big\}.
\end{align*}

For a compact set $F\subset \mathbb{R}^d$ we have
$$0\leq \dim_LF\leq \dimh F\leq \da F\leq \uda F\leq d.$$

The lower dimension is sensitive to the local structure of the set, for instance if $F$ has an isolated point then $\dim_L F=0$ and if $F$ is self-similar set then $\dim_L F=\dim_H F$. For more information on that dimension we refer to the Fraser's book \cite{fraser2020assouad}.

\section{Existence of  $(\delta,s,t)-$Frostman measures}

Our main interest is to find conditions on the set $E$ such that there exist a family of measures \(\{\mu_\delta\}\) that simultaneously satisfy both decay estimates \eqref{Hausdorff} and \eqref{FrostmanIneq}.

In order to be able to construct them, we first need to introduce a new parameter, which gives us information about the distribution of the set when refining dyadic cubes. Recall that $\#A$ is the number of elements of the set $A$.
\begin{defi}
    We define the \textit{Dyadic dimension} as follows:
    \begin{align*}\mathcal{D}(E):=\liminf_{n\to\infty}   
     \frac{\log\mathcal{N}_n(E)}{\log2}
    \end{align*}
\end{defi}
where $\mathcal{N}_n(E)=\displaystyle\min_{Q\in \mathcal{D}_n} \#\{Q'\in \mathcal{D}_{n+1} :\ Q'\cap( E\cap Q)\neq\emptyset\}$.

\begin{theo}\label{frostam modificado}
    Let $\theta >0$ and $E\subset \mathbb{R}\pot{d}$ be a compact set, such that $0<\mathcal{D}(E)$. Then for all $0<t<\uda E$ and all $0<s<\mathcal{D}(E)$ with $s\leq t$, there exists $c>0$ such that for all $\delta_0>0$ exist $\delta\in(0,\delta_0)$ and a Radon measure $\mu_{\delta}$ , supported on $E$,  satisfying
    \be\label{aclarar3}\mu_{\delta}(B(x,r))\leq  \left\{ \begin{array}{lcc} c (\delta\pot{1/\theta})\pot{t-s}r\pot{s} & if & r\in(0,\delta\pot{1/\theta}) \\ c r\pot{t} & if & \delta\in[\delta\pot{1/\theta},\delta] \end{array} \right.\ee
    for all $x\in E$.
    
    Equivalently,  for all $0<t<\da E$ and all $0<s<\mathcal{D}(E)$ with $s\leq t$, , there exists $c>0$ and $\delta_0>0$ such that for all $\delta\in (0,\delta_0]$ there exists a measure $\mu_\delta$ supported on $E$ satisfying \eqref{aclarar3}.
\end{theo}

\begin{proof}

If $s=t$ the result follows by choosing the measure from the classical Frostman lemma for Hausdorff dimension \eqref{Hausdorff}.
Let $0<s<\mathcal{D}(E)$ and $N$ sufficiently large such that $2^s<\mathcal{N}_n(E)$ whenever $n\geq N$. 
Let $\mathcal{D}_m$ ($m\geq N$), the dyadic partition of $[0,1]\pot{d}$, i.e. , $2\pot{dm}$ pairwise disjoint half-open dyadic cubes of side length $2\pot{-m}$.  Without loss of generality we may assume that $E\subset [0,1]\pot{d}$ and that $E$ is not contained in any $Q\in\mathcal{D}_1$.

Let $\delta_0$ be sufficiently small such that $\delta_02\pot{N}<1$.

Let $t<\uda E$. From the definition of $\uda E$ we have that there exists $\epsilon>0$ and a decreasing sequence $\{\delta_k\}_{k\in\mathbb{N}}$ with $\delta_1 <\delta_0$ and $\delta_k \rarw 0$ when $k\rarw\infty$ such that for all covers $\{U_i\}_i$ of $E$ with $\delta_k\pot{1/\theta}\leq |U_i|\leq \delta_k$, we have

\begin{equation}\label{mu2}
    \displaystyle\sum_i |U_i|\pot{t}>\epsilon.
\end{equation}

Given $k$, let $m\geq 0$ be the unique integer satisfying $2\pot{-m-1}<\delta_k\pot{1/\theta}\leq 2\pot{-m}$. 

The idea is to assign to each cube $Q\in \mathcal{D}_m$ that intersects $E$, a mass $\mu_m (Q)= 2\pot{-mt}$ and then distribute it uniformly into the cubes in $\mathcal{D}_{m+1}$ that intersect $E$ and so on.

For this, let $Q\in \mathcal{D}_n$ with $n\geq m+1$ and let $Q\pot{*_i}$ the unique cube in $\mathcal{D}_{n-i}$ that contains $Q$. Define $\# \mathcal{Q}_{n}(Q)$ as the number of cubes in $\mathcal{D}_n$ that intersect $E\cap Q\pot{*_1}$. Let us denote $Q\pot{*_0}=Q$, and for $Q\in \mathcal{D}_n$ define $\displaystyle\Phi_{m+1}(Q)=\displaystyle\Pi_{i=0}\pot{n-(m+1)}\#\mathcal{Q}_{n-i}(Q\pot{*_i})$.

Basically for $Q\in \mathcal{D}_n$, the parameter $\Phi_{m+1}(Q)$ is an idea of how much of the set $E\cap Q'$, where $Q'$ is the unique cube $Q'\in\mathcal{D}_m$, such that  $Q\subset Q'$, is contained in the different levels from $\mathcal{D}_m$ until $\mathcal{D}_n$.

With this in mind, we define a measure $\mu_m$ on $\mathbb{R}\pot{d}$ such that for any $n\geq m+1$ and $Q\in \mathcal{D}_n$

$$\mu_m  (Q)= \left\{ \begin{array}{lcc} 2\pot{-mt}/\Phi_{m+1} (Q) & if & Q\cap E\neq\emptyset \\ 0 & if & Q\cap E =\emptyset.\end{array} \right.$$

For  each $Q\pot{*}\in\mathcal{D}_m$ that intersects $E$ we have

$$\mu_m(Q\pot{*})=\displaystyle\sum_{\{Q\in\mathcal{D}_{m+1}:Q\subset Q\pot{*}\}} \mu_m(Q)=\Phi_{m+1}(Q) 2\pot{-mt}/\Phi_{m+1}(Q) =2\pot{-mt}.$$

Using this and recalling that $Q\pot{*i}$ is the unique cube in $\mathcal{D}_{n-i}$ that contains $Q$, we have for any cube $Q\in\mathcal{D}_n$ ($n\geq m+1$):

\begin{equation}\label{muineq}
\begin{split}
    \mu_m(Q)=\frac{2\pot{-mt}}{\Phi_{m+1}(Q)} &=\frac{\mu_m(Q\pot{*(n-m)})}{\Phi_{m+1}(Q)}\\
    &=\frac{\mu_m(Q\pot{*(n-m)})}{\Phi_{m+1}(Q)} \frac{\#\mathcal{Q}_m(Q\pot{*(n-m)})}{\#\mathcal{Q}_m(Q\pot{*(n-m)})}\\
    &=\frac{\mu_m(Q\pot{*(n-(m-1))})}{\Phi_m(Q)}.
    \end{split}
\end{equation}

The next step is to modify the measure $\mu_m$ to obtain a measure $\mu_{m-1}$.

For $Q\in \mathcal{D}_{n}$ with $n\geq m $ let

$$\mu_{m-1}(Q) = \left\{ \begin{array}{lcc} 2\pot{-(m-1)t}/\Phi_m (Q) & if &  \mu_m (Q\pot{*(n-(m-1))})>2\pot{-(m-1)t} \\ \mu_m (Q) & if & \mu_m (Q\pot{*(n-(m-1))})\leq 2\pot{-(m-1)t}.\end{array} \right.$$

Again we have for  $Q\pot{*}\in\mathcal{D}_{m-1}$ that intersects $E$:

\ 

If $\mu_m (Q\pot{*})\leq 2\pot{-(m-1)t}$ then 

$$\mu_{m-1}(Q\pot{*})=\sum_{Q\in\mathcal{D}_m :Q\subset Q\pot{*}}\mu_{m-1}(Q)=\sum_{Q\in\mathcal{D}_m :Q\subset Q\pot{*}}\mu_{m}(Q)=\mu_m (Q\pot{*})\leq 2\pot{-(m-1)t}.$$

\

And if $\mu_m (Q\pot{*})> 2\pot{-(m-1)t}$ then 

$$\mu_{m-1}(Q\pot{*})=\sum_{Q\in\mathcal{D}_m :Q\subset Q\pot{*}}\mu_{m-1}(Q)= \sum_{Q\in\mathcal{D}_m :Q\subset Q\pot{*}}2\pot{-(m-1)t}/\Phi_m (Q)=2\pot{-(m-1)t}.$$
\vspace{0.5 cm}

Then for $Q\in \mathcal{D}_{m-1}$\\
\begin{equation*}
    \mu_{m-1}(Q)\leq 2\pot{-(m-1)t} .
\end{equation*}

In the same way, using \eqref{muineq}, we have for $Q\in \mathcal{D}_n$, $n\geq m$:

If $\mu_m (Q\pot{*(n-(m-1))})\leq 2\pot{-(m-1)t}$ then $\mu_{m-1}(Q)=\mu_m(Q)$

and if $\mu_m (Q\pot{*(n-(m-1))})> 2\pot{-(m-1)t}$

$$\mu_{m-1}(Q)=\frac{2\pot{-(m-1)t}}{\Phi_m(Q)}<\frac{\mu_m (Q\pot{*(n-(m-1))})}{\Phi_m(Q)}=\mu_m (Q),$$
and therefore

$$\mu_{m-1} (Q)\leq \mu_m (Q), \quad \text{ for any } Q\in\mathcal{D}_n\ , n\geq m-1.$$

Proceeding inductively, given $\mu_{m-k}$ construct a measure $\mu_{m-k-1}$ such that for $Q\in\mathcal{D}_n$ with $n\geq m-k $

$$\mu_{m-k-1}(Q) = \left\{ \begin{array}{lcc} 2\pot{-(m-k-1)t}/\Phi_{m-k} (Q) & if &  \mu_{m-k} (Q\pot{*(n-(m-k-1))})>2\pot{-(m-k-1)t} \\ \mu_{m-k} (Q) & if & \mu_{m-k} (Q\pot{*(n-(m-k-1))})\leq 2\pot{-(m-k-1)t}.\end{array} \right.$$

Terminate this process with $\mu_{m-\ell}$ where $\ell$ is the largest
integer satisfying $2\pot{-(m-l)}n\pot{1/2}\leq\delta$, so for each cube $Q\in\mathcal{D}_{m-\ell}$ we have 
$|Q|=2\pot{-(m-\ell)}n\pot{1/2}\leq\delta $.

By construction we have

\begin{equation}\label{mu}
    \mu_{m-\ell}(Q_i)\leq 2\pot{-(m-i)t} \hspace{1 cm} Q_i\in\mathcal{D}_{m-i},\ i=0,...,\ell
\end{equation}
and

\begin{equation}\label{ineqmu}
    \mu_{m-\ell}(\cdot)\leq\mu_{m-\ell+1}(\cdot)\leq \dotsi \leq \mu_m (\cdot)
\end{equation}

Further, if a cube $Q$ satisfies an equality in \eqref{mu} for $\mu_{m-i}$ then either $Q$ or $Q\pot{*1}$ satisfies the equality for $\mu_{m-i-1}$.

Since all the dyadic cubes in $\mathcal{D}_m$ satisfy $\mu_m(Q)=2\pot{-m}$, we have that for all $x\in E$, there exists $0\leq i\leq \ell$ such that for some cube $Q\in\mathcal{D}_{m-i}$ that contains $x$, \eqref{mu} is an equality.

For each $x\in E$, choose the largest cube $Q$ that contains $x$ and satisfies the equality in $\eqref{mu}$. This process yields a finite collection of cubes $Q_1 , Q_2,... Q_p$ which cover $E$ and satisfy $\delta\pot{1/\theta}\leq |Q_i|\leq \delta$ for $i=1,2...,p$.

Now, using (\ref{mu2}) we have

$$\mu_{m-\ell}(E)=\displaystyle\sum_{i=1}\pot{p}\mu_{m-\ell}(Q_i)=\sum_{i=1}\pot{p} |Q_i|\pot{t}d\pot{-t/2}>\epsilon d\pot{-t/2}.$$

Define $\mu_{\delta_k}(\cdot) = \mu_{m-\ell}(E)\pot{-1}\mu_{m-\ell}(\cdot)$.

Finally, let $\sigma(\mathcal{D})$ the $\sigma$ algebra generated by $\mathcal{D}=\bigcup_{n\in\mathbb{N}}\mathcal{D}_n$, and then, if $A\subset\rd$, we can extend $\mu_{\delta_k}$ to $\rd$ by $\mu_{\delta_k}(A)=\inf \left\{ \mu_{\delta_k}(B):A\subset B\in \sigma(\mathcal{D})\right\}$.

Now we will see that $\mu_{\delta_k}$ satisfies the inequality \eqref{aclarar3}.

Let $x\in E$, $\delta_k\pot{1/\theta}\leq r\leq \delta$, then if $k$ is chosen as the unique integer satisfying $2\pot{-(m-k+1)}< r\leq 2\pot{-(m-k)}$, we have that the ball $B(x,r)$ is contained in $c_d$ cubes in $\mathcal{D}_{m-k}$ where $c_d$ is a constant that depends only on $d$.

Then, again by (\ref{mu}) we have

$$\mu_{\delta_k}(B(x,r))\leq c_d \mu_{m-l}(E)\pot{-1}2\pot{-(m-k)t}\leq c_d \epsilon\pot{-1} d\pot{t/2} 2\pot{t} r\pot{t}.$$

Therefore we have that there exists $C$ only depending on $d$ and $t$ such that  for all $\delta_k\pot{1/\theta}\leq r\leq \delta_k$ and $x\in E$ 

\begin{equation}\label{frostman}
    \mu_{\delta_k}(B(x,r))\leq Cr\pot{t}.
\end{equation}

Now, by our choice of $s$ and $\delta_0$ we have that for each cube $Q\in \mathcal{D}_n$, with $n\geq m+1$,

\begin{equation}\label{muineqineq}
    \,2\pot{s(n-m)}\leq \Phi_{m+1}(Q).
\end{equation}

Let $x\in E$ and $r<\delta_k\pot{1/\theta}$ then if $n'$ is the unique integer satisfying $2\pot{-n'-1}< r \leq 2\pot{-n'}$ we have that the ball $B(x,r)$ is contained in at most $c_d$ cubes in $\mathcal{D}_{n'}$ and therefore

$$\mu_{\delta_k}(B(x,r))\leq c_d \max \{ \mu_{\delta_k}(Q) : Q\in\mathcal{D}_{n'}\}=c_d \mu_{\delta_k}(Q_{n'}),$$

for some $Q_{n'}\in \mathcal{D}_{n'}$. Now using \eqref{ineqmu} and \eqref{muineqineq} we have

\begin{equation}\label{frostman N2}
    \begin{split}
        \mu_{\delta_k}(B(x,r))\leq c_d \mu_{\delta_k}(Q_{n'}) &\leq c_d \frac{\mu_m (Q_{n'}\pot{*(n'-m)})}{\Phi_{m+1}(Q_{n'})}\\
        &\leq c\, c_d \frac{2\pot{-mt}}{2\pot{s (n'-m)}}\\
        &\leq c\, c_d(2\pot{-m})\pot{t-s}(2\pot{-n'})\pot{s}\\
        &\leq C' (\delta_k\pot{1/\theta})\pot{t-s}r\pot{s},
    \end{split}
\end{equation}

with $C'=c\,c_d 2\pot{t-s}$.

Then, combining \eqref{frostman} and \eqref{frostman N2} the result follows.

 The case of $\da E$ is similar.
\end{proof}

By the Frostman lemma (see \eqref{Hausdorff}), the optimal situation for constructing a measure satisfying the decay estimate \eqref{aaclarar3} corresponds to the case $s \leq \dim_H E$. 

It is not hard to show that, for any compact set $E$, one has
\[
\mathcal{D}(E) \leq \dim_L E.
\]  

However, as far as we know, the reverse inequality does not necessarily hold. This observation naturally leads to the following question.

\begin{question}
Is it true that for all \(s < \dim_L E\) (or even in the optimal case \(s < \dim_H E\)) and \(0 < t < \overline{\dim}_\theta E\) with \(t \leq s\), there exists a constant \(c > 0\) such that for every \(\delta_0 > 0\) there exist \(\delta \in (0, \delta_0)\) and a Radon measure \(\mu_\delta\) supported on \(E\) satisfying \eqref{aaclarar3}?\\[0.3em]
Moreover, does the same conclusion hold for all sufficiently small \(\delta > 0\) if we replace \(\overline{\dim}_\theta E\) with \(\underline{\dim}_\theta E\)?
\end{question}

It seems plausible that the result could be further improved by refining the argument, replacing the dyadic decomposition with more general partitions of the ambient space. Exploring this direction could be an interesting avenue for future work.

\section*{Funding}
The research of the authors is partially supported by Grants
PICT 2022-4875 (ANPCyT) (no disbursements since December 2023), 
PIP 202287/22 (CONICET), and 
UBACyT 2022-154 (UBA). Nicolas Angelini is also partially supported by PROICO 3-0720 ``An\'alisis Real y Funcional. Ec. Diferenciales''.

\bibliographystyle{plain} 
\bibliography{ref}
\end{document}